\documentclass[12pt]{article}
\usepackage{graphicx,lpic}

\usepackage{mathrsfs}
\usepackage{amsmath, amscd, amsthm,amssymb, amsfonts, verbatim,subfigure, slashed,stmaryrd}
\usepackage[mathcal]{eucal}
\usepackage{braket}     % brakets and sets

\usepackage[colorlinks=true,linkcolor=red,citecolor=blue]{hyperref}

\usepackage[latin1]{inputenc}

\usepackage{epsfig,color}

\usepackage{pb-diagram}

\renewcommand{\epsilon}{\varepsilon}

\newcommand{\C}{\mathbb C}
\newcommand{\N}{\mathbb N}

\newcommand{\op}{\operatorname}
\newcommand{\mbf}{\mathbf}
\newcommand{\mbb}{\mathbb}
\newcommand{\mc}{\mathcal}

\newcommand{\abs}[1]{\left| #1 \right|}

\newcommand{\R}{\mbb R}

\newcommand{\SC}[1]{\left\llbracket #1\right\rrbracket}

\newcommand{\bds}{\boldsymbol}

\theoremstyle{thm}

\newtheorem{theorem}{Theorem}[section]
\newtheorem{thm-def}{Theorem/Definition}[theorem]
\newtheorem{proposition}[theorem]{Proposition}

\newtheorem{lemma}[theorem]{Lemma}

\theoremstyle{definition}
\newtheorem{example}[theorem]{Example}
\newtheorem{remark}{Remark}[section]

\newtheorem*{examples}{Examples}

\numberwithin{equation}{subsection}

\theoremstyle{definition}
\newtheorem{definition}[theorem]{Definition}

\theoremstyle{rem}

\usepackage{tikz}

\title{Oriented cobordism of random manifolds}

\author{Miguel Berm\'udez\thanks{\texttt{bermudez@math.univ-paris-diderot.fr}}\medskip
\\{\small Universit\'e Paris Diderot}
\\{\footnotesize B\^atiment Sophie Germain},
{\footnotesize Case 7012},
{\footnotesize 75205 Paris Cedex 13}
}
\date{}

\begin{document}
\maketitle

\begin{abstract}
We introduce a general framework allowing the systematic study of random manifolds. In order to do so, we will put ourselves in a more general context than usual by allowing the underlying probability space to be non commutative in the sens of Connes \cite{connes80surveyfoliations}. We introduce in this paper the oriented cobordism groups of random manifolds, which we compute in dimensions $0$ and $1$, and we prove the surjectivity of the corresponding Thom-Pontryagin homomorphism. Non commutative and usual (commutative) random manifolds are naturally related in this setting since two commutative random manifolds can be cobordant through a non commutative one, even if they are not cobordant as usual random manifolds. The main interest of our approach is that expected characteristic numbers can be generalized to the non commutative case and remain naturally invariant up to cobordism. This includes Hirzebruch signature, Pontryagin numbers and, more generally, the expected index of random elliptic differential operators. 
\end{abstract}

\section{Introduction}
A random manifold is a standard probability space $\mbf K$ whose elements are connected manifolds. We define the total space of a random manifold $\mbf K$ as the disjoint union $\mbf X=\coprod_{M\in \mbf K}M$. We shall see $\mbf X$ as a laminated set whose leaves are the elements of $\mbf K$, and assume that there is a standard measurable structure on $\mbf X$ compatible with the usual Borel structure on the leaves and on $\mbf K$. 

%In most these cases the underlying random manifold $\mbf K$ is not a standard probability space and one prefers then to replace the measurable space $\mbf K$ by a discrete standard subset $T\subset \mbf X$ intersecting every leaf of $\mbf X$, and we shall assume that the equivalence relation on $T$ induced by the leaves of $\mbf X$ is standard in the sens of \cite{Moore and alt} and that there is a probability measure on $T$ which is invariant in the sens of \cite{Gab,Gab}. In other words, we can see $\mbf K$ as a non commutative probability space in the sens of Connes \cite{Con}. In general, that case, we shall call $\mbf X$ a singular or non commutative random manifold.

We call {\em observable} every map $b:\mbf K\to\R$ that associates to every manifold in $\mbf K$ a real number. For instance: the dimension, the Betti numbers, the Euler characteristic, the Hirzebruch signature and, more generally, the index of differential and pseudo-differential operators, are examples of such invariants. In the setting of random geometry, the most natural question is the following: what is the expected value of $b$ in $\mbf K$? Of course, in order to get a well posed question, we must guarantee that the map $b$ is measurable and, more precisely, integrable. For classical invariants as those mentioned above, it deeply depends on the measurable structure and specific probability measure we put on the set $\mbf K$. Our main goal is to develop a general framework allowing the systematic study of such invariants. 

Our first step will be to enlarge the notion of random manifold by allowing the probability space $\mbf K$ to be non commutative in the sense of Connes \cite{connes90book}. A non commutative probability space is defined by Connes as (the Morita equivalence class of) a separable von Neumann algebra $\mc N(\mbf K)$ endowed with a weakly continuous tracial state $\tau$ \cite{connes80surveyfoliations}. Invariants can be viewed as elements of $\mc N(\mbf K)$ and their expected value is given by the trace. If $\mbf K$ is a classical probability space then $\mc N(\mbf K)=L^{\infty}(\mbf K)$, i.e. a commutative von Neumann algebra, and the trace is the usual Lebesgue integral. 

Random manifolds arise naturally in quantum physics. Up to a Wick rotation, which formaly replaces the real time coordinate $t$ by a pure imaginary time $\tau=it$, the space-time becomes riemannian and quantum mechanics reduce to random geometry. For instance, a Wick rotated relativistic quantum particle is given by a $4$-dimensional manifold $M$ (representing the space-time) together with a set $\mbf K$ of continuous paths $\alpha:[0,1]\to M$. The probability measure on $\mbf K$ is the Wiener limit of random walks associated to random jumps of probability density
$$
p(x|y)\propto e^{-d(x,y)}d\op{vol}_M(x),
$$
where $d(x,y)$ is the metric distance and $d\op{vol}_M$ the metric volume on $M$. Another appealing example is string theory, which is a $2$-dimensional version of the relativistic particle above. It is obtained, roughly speaking, by replacing random paths by random immersions $\alpha:\Sigma\to M$ with $\Sigma$ a surface with boundary, random walks by PL immersions, and the metric length of the path by the metric area of the surface. These two examples are particular cases of a more general set of physical theories called brane theories. In general, brane theories can be viewed as bordisms in the category of random manifolds. See \cite{polchinski1998string} for a complete introduction to string theory. Another interesting example is pure quantum gravity, which is given by a random riemannian bordism of a fixed compact riemannian $3$-manifold $M$, whose expected value is an Einstein manifold. Since string theory is given by a random $2$-dimensional submanifold of a fixed manifold, the construction can be made rigorous by using the classification of compact surfaces and PL immersions. In turn, quantum gravity deals with the set of all $4$-dimensional riemannian manifolds whose boundary is $M$, and there is no natural way to put a probability measure on this set, mostly because we don't know what it looks like. The usual approach is to fix the underlying topology and to randomize the metric. Since a riemannian metric can be viewed as a field of $4\times 4$ non singular positive matrices, this reduces quantum gravity to a most familiar quantum field theory, but the theory still has far too many symmetries to be renormalizable. However, quantum gravity remains an extremely speculative domain, and no satisfactory rigorous mathematical model has been found so far. One of the main goals of our introduction to the general study of random manifolds is to clarify and unify the existing approaches to random geometry.

A very fruitful setting is given by algebraic random $n$-manifolds, which can be defined as $p^{-1}(0)$ where $p:\C^{n+m}\to \C^p$ is a random polynomial. In that case the probability space $\mbf K$ is a finite vector space of polynomials and can therefore be endowed with a family of natural probability measures. In a recent work Gayet and Welschinger make very stimulating progresses on the computation of expected Betti numbers in this setting \cite{gayetwelschinger2016betti}.

Given an oriented random manifold with boundary $\mbf K$, choosing a random connected component of the boundary of an element of $\mbf K$ gives rise to a random manifold, provided that we can measurably associate to every element $M\in \mbf K$ a probability measure on the set of connected components of $\partial M$. The corresponding random manifold will be called the boundary of $\mbf K$. Two random manifolds $\mbf K$ and $\mbf K'$ are said to be {\em cobordant} if their disjoint union $\mbf K+\mbf K'$ is isomorphic to the boundary of a random manifold. Cobordism obviously defines an equivalence relation on the set of random manifolds and one can easily see that the disjoint union induces a structure of abelian group on the set of cobordism classes. The aim of this paper is to introduce the study of these groups. 
 
\section{Random manifolds}
A {\em BT-space} is given by a standard Borel space $\mbf X$ endowed with a topology whose connected components, endowed with its Borel $\sigma$-algebra are measurable subspaces of $\mbf X$. A map $\mbf X\to \mbf Y$ is said to a \textit{BT-map} if it is a simultaneously measurable and continuous map that sends transversals of $\mbf X$ into transversals of $\mbf Y$. The class of BT-spaces together with BT-maps is a category called \textbf{BTop}. We call \textit{leaves} of a BT-space the connected components of the underlying topology and \textit{transversals} discrete Borel subsets. Two transversals $S$ and $T$ are said to be isomorphic if there exists a leaf-preserving Borel isomorphism $\gamma:T\to S$ that we call a \textit{holonomy transformation}.

\begin{examples}
\begin{enumerate}
  \item A connected Polish space $V$, endowed with is Borel $\sigma$-algebra, is an example of BT-space of a single leaf and transversals, which are simply discrete subsets of $V$, are always countable. 
  \item Another example is given by a standard Borel space $\mbf K$ endowed with the discrete topology. In that case every Borel subset of $\mbf K$ is a transversal. 
  \item Given two BT-spaces $\mbf X$ and $\mbf Y$, we shall note $\mbf X\otimes \mbf Y$ the cartesian product of $\mbf X$ and $\mbf Y$ endowed with the product topology and the product measurable structure. 
  \item The BT-space $V\otimes\mbf K$, obtained as a product of a Polish space $V$ and a standard Borel space $\mbf K$ will be called a {\em prism} of \textit{base} $V$ and \textit{vertical} $\mbf K$.
\end{enumerate}	
\end{examples}

\begin{definition}[\cite{bermudez2006char}]
	Let $\mbf X$ a BT-space. An {\em transverse invariant measure} on $\mbf X$ is a map that associates to every transversal $T$ of $\mbf X$ a $\sigma$-finite Borel measure $\mu_T$ on $T$ such that $\gamma^*\mu_T=\mu_S$ for every holonomy transformation $\gamma:S\to T$. The pair $(\mbf X,\mu)$ will be called a {\em random topological space}. A random topological space is said to be a \textit{random manifold (with boundary)} if all its leaves are separable topological manifolds (with boundary). Let $(\mbf X,\mu)$ and $(\mbf Y,\nu)$ two random topological spaces. A BT-map $f:\mbf X\to \mbf Y$ is said to {\em measure preserving} if for every transversal $T$ of $\mbf X$ we have
	$$
	\mu(T)=\int_{f(T)}\abs{f^{-1}(t)}d\nu(t)
	$$
	where $\abs{\cdot}$ denotes the cardinal of a set. The category of random topological spaces together with measure preserving BT-maps will be called \textbf{RTop}.
\end{definition}

\begin{examples}
\begin{enumerate}
  \item If $V$ is a Polish space, two discrete subsets of $V$ are isomorphic as transversals if and only if they have the same cardinality; hence the counting measure gives rise to an invariant transverse measure. 
  \item Another example is given by a standard measure space $\mbf K$ endowed with the discrete topology. In that case every Borel subset of $\mbf K$ is a transversal and the only holonomy trasformation is the identity. Hence, every Borel measure on $\mbf K$ induces a transverse measure which is automatically invariant. 
  \item Given two random topological spaces $(\mbf X,\mu)$ and $(\mbf Y,\nu)$, the product of a transversal of $\mbf X$ and a transversal of $\mbf Y$ is a transversal of $\mbf X\otimes\mbf Y$ and one can easily verify that the corresponding product measure $\mu\otimes\nu$ is invariant. 
\end{enumerate}

\end{examples}
 Let $(\mbf X,\mu)$ be a random manifold (with boundary). We call \textit{atlas} of $\mbf X$ a countable set of isomorphisms of BT-spaces
 $$
 \mathfrak A=\{\varphi_{i}:U_{i}\to V_{i}\otimes \mbf K_i\}_{i\in I}
 $$
 called \textit{local charts}, where $\{U_i\}_{i\in I}$ is an open Borel covering of $\mbf X$, $V_i$ is a open subset of the euclidean space $\R^n$ (resp. the half-space $\mathbb H^n$) and $\mbf K_i$ is a standard Borel space for every $i\in I$. Notice that the change of charts $\varphi_{i}\circ\varphi_{j}^{-1}$ has automatically the form
\begin{equation*}
\varphi_{i}\circ\varphi_{j}^{-1}(x,t)=(f_{t}(x),\gamma(t))
\end{equation*}
where $\gamma$ is a partial Borel isomorphism $\mbf K_j\to \mbf K_i$ and $f_{t}$ is a local homeomorphism of $\R^{n}$ (resp. $\mathbb H^n$). An atlas is said to be {\em oriented} if the local homeomorphisms $f_t$ preserve the orientation. Notice that the invariant measure $\mu$ induce a measure on each $\mbf K_i$ preserved by change of verticals $\gamma$. We define the {\em cost} of $\mathfrak A$ as 
$$
c(\mathfrak A)=\sum_{i\in I}\mu(\mbf K_i)
$$

Let $V$ be an open subset of $\R^n$ and $\mbf X=V\otimes \mbf K$ a trivial random manifold of base $V$. We shall note $\mc A(V\otimes \mbf K)=C^{\infty}(V,L^{\infty}(\mbf K))$, the commutative algebra of smooth maps from $V$ to the Banach algebra $L^{\infty}(\mbf K)$ of essentially bounded functions on $\mbf K$. Every element $a\in \mc A(V\otimes\mbf K)$ can be viewed as a measurable continuous fonction $a:V\times\mbf K\to \R$, two such functions representing the same element of $\mc A(V\otimes\mbf K)$ if they coincide on almost every plaque $V\times\{\bullet\}$. For every measurable subset $B\subset V\times \mbf K$, we denote by $\mc A(B)$ the algebra formed by the restrictions to $B$ of the elements of $\mc A(V\otimes \mbf K)$. A map $f:B\to B'$ between Borel subsets of trivial random manifolds is said to be {\em smooth} if for every $a\in \mc A(B')$ the fonction $a\circ f$ belongs to $\mc A(B)$. In that case the map $a\mapsto \mc A(f)(a):=a\circ f$ gives rise to a morphism of algebras $\mc A(f):\mc A(B')\to \mc A(B)$.  

Let $\mbf X$ be a random manifold. An atlas $\{\varphi_i:U_i\to V_i\otimes \mbf K_i\}_{i\in I}$ of $\mbf X$ is said to be \textit{smooth} if for every $i,j\in I$ the corresponding change of charts $\varphi_i\circ\varphi_j^{-1}:\varphi_j(U_i\cap U_j)\to \varphi_i(U_i\cap U_j)$ is smooth. Two smooth atlas are called compatible if their union is a smooth atlas. The equivalence class of compatible bounded atlas is called a \textit{smooth  structure} on $\mbf X$. A random manifold together with a smooth structure is called a \textit{smooth random manifold}. A smooth atlas $\{\varphi_i:U_i\to V_i\times \mbf K_i\}_{i\in I}$ is said to be {\em compact} if for every $i\in I$ there exists a compact subset $F_i\subset V_i$ such that the sets $\varphi_i^{-1}(F_i\times \mbf K_i)$ cover $\mbf X$.

\begin{definition}
	A random manifold will be called \textit{compact} if it has a compact smooth atlas of finite cost.
\end{definition}

\subsection{The tangent space}
Let $\mbf X$ be a smooth random manifold. We shall note
$$
\mc A(\mbf X)=\{a:\mbf X\to \R|\forall i\in I,\; a\circ \varphi_i^{-1}\in \mc A(V_i\times \mbf K_i)\}
$$

A tangent vector field of $\mbf X$ is by definition a derivation of the algebra $\mc A(\mbf X)$. We shall note $\tau(\mbf X)=\op{Der}\mc A(\mbf X)$ the vector space of tangent vector fields and we will call it the \textit{tangent space} of $\mbf X$. Since $\mc A(\mbf X)$ is a commutative algebra, the tangent space $\tau(\mbf X)$ has a natural structure of $\mc A(\mbf X)$-module.

\begin{lemma}\label{lem:trivial-free}
	If $\mbf X=V\otimes\mbf K$ is a trivial random $n$-manifold, then $\tau(\mbf X)\simeq\mc A(\mbf X)^n$.
\end{lemma}
\begin{proof}
	Let $\mbf X=V\otimes \mbf K$. Recall that every function $a\in \mc A(\mbf X)$ is  a map $a:V\to L^{\infty}(\mbf K)$ with continuous partial derivatives of every order. Hence the $n$ partial derivatives of order $1$ are derivations of $\mc A(\mbf X)$. We shall prove that every derivation $\xi$ of $\mc A(\mbf X)$ can be written in an unique way as $\xi=\sum_i a_i\partial_i$ with $a_i\in \mc A(\mbf X)$.  It follows from the Taylor's theorem with integral remainder (which is valid for smooth functions of $n$ variables taking values in any Banach space) that for every $y\in V$ there exists a star convex open subset $V_y\subset V$ centered in $y$ such that for $x\in V_y$ we have  
		$$
		a(x)=a(y)+\sum_i (x_i-y_i)\partial_i a(y)+\sum_{i,j}(x_i-y_i)(x_j-y_j)B^{ij}_{y}(x)
		$$
		where 
	where $x_i$ and $y_i$ denote the coordinates of $x$ and $y$ and $B^{ij}_{y}$ is a function in $\mc A(\mbf X)$ defined by the formula
	$$
	B^{ij}_{y}(x)=\frac 12\int_0^1(1-s)^2\partial_{ij}a(y+s(x-y))ds
	$$
	
	 Fixing $y$, applying $\xi$ to the identity above and then evaluating in $y$ we get identity
	$$
	\xi(a)(y)=\sum_i \xi(x_i)(y)\partial_ia(y)
	$$
	valid for every $a\in \mc A(\mbf X)$ and every $y\in V$. Hence $\xi=\sum_i \xi(x_i)\partial_i$, as wanted.
\end{proof}

\begin{lemma}
	Let $\mbf X$ be a random manifold. Every compact atlas of $\mbf X$
	$$
	\{\varphi_i:U_i\to V_i\otimes \mbf K_i\}_{i\in I}
	$$
	 admits a subordinated partition of unity, i.e. for every $i\in I$ there exists a positive element $a_i\in \mc A(\mbf X)$ with support in $U_i$ such that $\sum_{i\in I}a_i=1$.
\end{lemma}
\begin{proof}
	Let us consider a family of compact sets $F_i\subset V_i$ verifying the definition of compact atlas, and let $g_i$ a bump function on $V_i$ such that $F_i$ is contained in the interior of its support. If we define $h_i:U_i\to \R$ as the pull-back of $g_i$ by the projections $U_i\to V_i$, we obtain the wanted partition of unity by setting $a_i=h_i/\sum_i h_i$. 
\end{proof}

\begin{lemma}\label{lem:riemann-metric}
	If $\mbf X$ is a compact random manifold, it has a riemannian metric, i.e. a morphism of modules
	$$
	g:\tau(\mbf X)\otimes\tau(\mbf X)\to \mc A(\mbf X)
	$$
	such that for every $\xi,\nu\in \tau(\mbf X)$:
	\begin{enumerate}
  \item $g(\xi,\nu)=g(\nu,\xi)$;
  \item $g(\xi,\xi)\geq 0$;
  \item $g(\xi,\xi)= 0$ iff $\xi=0$.
\end{enumerate}

\end{lemma}
\begin{proof}
	The lemma obviously holds for trivial random manifolds, since their tangent space is a free module. Let us consider a compact atlas $\mathfrak A=\{\varphi_{i}:U_{i}\to V_{i}\otimes \mbf K_i\}_{i\in I}$ of $\mbf X$ together with a subordinated partition of unity $\sum_{i\in I}a_i=1$. For every $\xi\in \tau(\mbf X)$ one can see $a_i^{1/2}\xi$ as a derivation on $\tau(U_i)$. If we fix a riemannian metric on e $g_i$ on each $U_i$, then we can write
	$$
	g(\xi,\nu)=\sum_i g_i\left(a_i^{1/2}\xi,a_i^{1/2}\nu\right)
	$$
	A straigthforward computation shows that $g$ is a riemannian metric on $\tau(\mbf X)$. 
\end{proof}

\begin{remark}\label{rem:restriction}
It follows from the lemmas above that every tangent vector $\xi\in \tau(\mbf X)$ can be viewed as a map $\xi:x\in \mbf X\mapsto \xi_x\in T_xL_x$ where $L_x$ is the leaf of $\mbf X$ which contains $x$ and $T_xL_x$ the usual tangent space of $L_x$ at $x$. Hence for every measurable open subset $U\subset \mbf X$, we have an homomorphisms of $\mc A(\mbf X)$-modules $r_U:\tau(\mbf X)\to \tau(U)$ given be the restriction $\xi\mapsto \xi_{|U}$. In the other hand, a riemannian metric $g$ can be viewed as a field of inner products 
$$
g:x\in \mbf X\mapsto (g_x:T_xL_x\otimes T_xL_x\to \R)
$$ 	
\end{remark}

\begin{theorem}
For every compact random manifold $\mbf X$, $\tau(\mbf X)$ is a finitely generated projectif module over $\mc A(\mbf X)$.
\end{theorem}
\begin{proof}
Let $\dim X=n$ and let $\mathfrak A=\{\varphi_i:U_i\to V_i\otimes \mbf K_i\}_{i=1}^k$ be a compact smooth atlas of $\mbf X$ together with a subordinated partition of unity $\sum_ia_i=1$. Since
	$$
	a_i^{1/2} \tau(\mbf X)=a_i^{1/2} \tau(U_i)\simeq a_i^{1/2}\mc A(U_i)^n=a_i^{1/2}\mc A(\mbf X)^n\subset \mc A(\mbf X)^n
	$$ 
	the map 
	$$
	\xi\in \tau(\mbf X)\mapsto (a_1^{1/2}\xi,\cdots,a_k^{1/2}\xi)\in \mc A(\mbf X)^{kn}
	$$
	is injective and puts $\tau(\mbf X)$ as a sub-module of $\mc A(\mbf X)^{kn}$.  A straightforward computation shows that the orthogonal complement of $\tau(\mbf X)$ in $\mc A(\mbf X)^{kn}$ with respect to the usual $\mc A(\mbf X)$-inner product $\braket{a|b}=\sum a_ib_i$ is a $\mc A(\mbf X)$-module $\tau(\mbf X)^{\perp}$ verifying $\tau(\mbf X)\oplus \tau(\mbf X)^{\perp}=\mc A(\mbf X)^{kn}$.  
	\end{proof}

\section{Low dimesional cobordism groups}
In the sequel, every random manifold shall be supposed to be compact. The boundary of a random manifold with boundary $\mbf X$ is a random manifold without boundary that shall be noted $\bds \partial \mbf X$. If $\mbf X$ is oriented, its boundary inherits a natural orientation. Given an oriented random manifold $\mbf X$, we shall note $-\mbf X$ the same random manifold endowed with the opposite orientation. Given two random manifolds without boundary $\mbf X$ and $\mbf Y$, we shall note $\mbf X+\mbf Y$ their disjoint union and $\mbf X-\mbf Y=\mbf X+(-\mbf Y)$. We say that $\mbf X$ and $\mbf Y$ are cobordant if there exists an oriented random manifold with boundary $\mbf M$ such that $\bds \partial\mbf M=\mbf X-\mbf Y$. Cobordism defines an equivalence relation on the set of oriented random manifolds and we note $[\mbf X]$ the cobordism class of $\mbf X$. The set of cobordism classes of oriented random manifolds, together with the operation $[\mbf X]+[\mbf Y]=[\mbf X+\mbf Y]$ is an abelian group that we shall note $\mbf R\bds \Omega^{SO}$. It has a natural grading induced by the dimension, i.e. 
$$
\mbf R\bds \Omega^{SO}=\bigoplus_{n=0}^{\infty} \mbf R\bds \Omega^{SO}_n
$$
where $\mbf R\bds \Omega^{SO}_n$ is the group of cobordism classes of oriented random $n$-manifolds. The neutral element of $\mbf R\bds \Omega^{SO}$ is the cobordism class of the empty random manifold and will be noted $[\bds 0]$.

On can easily see that $\mbf R\bds \Omega_0\simeq \R$. Indeed, since an oriented connected $0$-manifold is signed point $x=\pm 1$, a random $0$-manifold is given by a finite standard space $\mbf K$ together with a measurable partition $\mbf K=\mbf K^++ \mbf K^-$. Moreover $[\mbf K]=[\bds 0]$ if and only if $\mu(\mbf K^+)=\mu(\mbf K^-)$. Hence the map $\mbf K\mapsto \mu(\mbf K^+)-\mu(\mbf K^-)$ induces an isomorphism of abelian groups $\Phi_0:\mbf R\bds\Omega^{SO}_0\to \R$.

\medskip
We devote the rest of this section to prove the following: 

\begin{theorem}
	Every oriented random $1$-manifold is the boundary of an oriented random $2$-manifold or, in other words, that
	$$
	\mbf R\bds\Omega^{SO}_1=0.
	$$
\end{theorem}

 We shall split the proof in several lemmas. Let us start by recalling some basic constructions and facts. Let $\mbf K$ a finite standard measure space. The {\em suspension} of a measure preserving automorphism $\gamma\in \op{Aut}(\mbf K)$ is the oriented random $1$-manifold $\mbf \Sigma_\gamma$ obtained from the trivial random $1$-manifold with boundary $[0,1]\times \mbf K$ by the following identifications
$$
(0,t)\sim (1,\gamma(t))\quad,\quad t\in \mbf K
$$

\begin{lemma}\label{compact2zero}
Let $\mbf X$ an oriented random $1$-manifold such that almost every leave of $\mbf X$ is compact. Then $[\mbf X]=[\bds 0]$.
\end{lemma}
\begin{proof}
Let $T$ be a transversal of $\mbf X$ and let us consider the standard equivalence relation $\mc R_T\subset T\times T$ whose equivalence classes are given by the intersection of $T$ with the leaves of $\mbf X$. Since the leaves are almost all compact, the equivalence classes of $\mc R_T$ are almost all finite. It is well known that such an equivalence relation has a fundamental domain, i.e. a Borel subset $\mbf K\subset T$ which intersects almost every leave in exactly one point, and one has an obvious isomorphism of random manifolds between $\mbf X$ and a random circle $\mbf X\simeq \mbb S\times \mbf K$ which is the boundary of $\mbb D\times \mbf K$.
\end{proof}

\begin{lemma}\label{every2noncompact}
 For every oriented random $1$-manifold $\mbf X$ there exists an oriented random $1$-manifold $\mbf X'$ without compact leaves such that $[\mbf X]=[\mbf X']$.
\end{lemma}
\begin{proof}
Let $\mbf X$ be a $1$-lamination and let $T$ be a transversal which intersects every leaf of $\mbf X$. Since $\mbf X$ is compact as random manifold, a leaf $L$ of $\mbf X$ is compact if and only if $T\cap L$ is finite. The finite classes of $\mc R_T$ form a Borel subset $\mbf K\subset T$, which implies that the compact leaves of $\mbf X$ form a random $1$-manifold $\mbf F\subset \mbf X$. Therefore we can write $\mbf X= \mbf F+\mbf X'$ where $\mbf X'=\mbf X\backslash\mbf F$ has no compact leaves by construction. By the lemma \ref{compact2zero} above we have
$$
[\mbf X]=[\mbf F]+ [\mbf X']= [\bds 0]+[\mbf X']=[\mbf X']
$$
\end{proof}

\begin{lemma}\label{every=suspension}
 For every oriented random $1$-manifold $\mbf X$ there exists a finite standard measured space  $\mbf K$ and a measure preserving automorphism $\gamma\in \op{Aut}(\mbf K)$ such that $\mbf X\simeq \mbf \Sigma_\gamma$.
 \end{lemma}
 \begin{proof}
 Let $\mbf K$ be a finite complete transversal of $\mbf X$. If we cut $\mbf X$ along $\mbf K$, we obtain a random $1$-manifold with boundary $\hat{\mbf X}$ whose almost every leaf is diffeomorphic to the interval $[0,1]$. As in the proof of lemma \ref{compact2zero}, $\hat{\mbf X}$ is diffeomorphic to the trivial $1$-manifold $[0,1]\times \mbf K$ in such a way that the quotient map $\hat{\mbf X}\to \mbf X$ preserves the orientation. Hence every point $(0,t)\in \hat{\mbf X}$ is identified to exactly one point of the form $(1,s)$, and the map $\gamma:t\in \mbf K\to s\in \mbf K$ is a measure preserving automorphism of $\mbf K$ such that $\mbf X\simeq \mbf \Sigma_\gamma$.
 \end{proof}

\begin{lemma}\label{homo}
For every finite standard measure space $\mbf K$, the map
\begin{equation*}
\gamma\in \op{Aut}(\mbf K)\mapsto \left[\mbf \Sigma_{\gamma}\right]\in \mbf R\bds\Omega^{SO}_{1}
\end{equation*}
is a homomorphism of groups.
\end{lemma}
\begin{proof}
Let $\phi$ and $\psi$ two measure preserving automorphisms of $\mbf K$. It suffices to prove that there exists a cobordism between the disjoint union $\mbf \Sigma_\phi+ \mbf \Sigma_\psi$ and $\mbf \Sigma_{\phi\circ\psi}$, i.e. an oriented random $2$-manifold with boundary $\mbf M$ such that
$$
\bds\partial\mbf M=\mbf \Sigma_\phi+ \mbf \Sigma_\psi-\mbf \Sigma_{\phi\circ\psi}
$$
. For the sake of clarity, we will explain graphically how it works. The random $2$-manifold $\mbf M$ is given as a fibration over the singular surface with boundary on the left of figure \ref{fig:r}, where the fibers are the random singular $1$-manifolds on the right. It is obvious from the picture that the singularities on the leaves of $\mbf M$ are of Morse type, which implies that the leaves of $\mbf M$ are smooth manifolds. The local triviality and the compacity of $\mbf M$ follows from the local triviality of the fibration. 
\end{proof}

\begin{figure}\label{fig:r}
\begin{center}
\includegraphics[width=10cm,height=8cm]{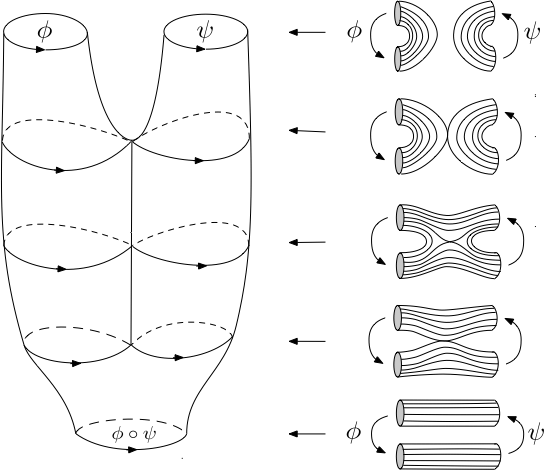}
\caption{Cobordism between $\bds \Sigma_\phi+\bds \Sigma_\psi$ and $\bds \Sigma_{\phi\circ\psi}$}	
\end{center}
\bigskip
\end{figure}

The combination of the above lemmas allows us to conclude. Indeed, since $\op{Aut}(\mbf K)$ is a simple group (see Fathi \cite{fathi78automorphimssimple}) and $\mbf R\bds \Omega^{SO}_1$ is abelian, every cobordism class on the image of the homomorphism of lemma \ref{homo} vanishes. Hence, by lemma \ref{every=suspension}, this will be the case for every cobordism class in $\mbf R\bds \Omega^{SO}_{1}$, which completes the proof.

\section{Higher dimensional cobordism groups}
We shall prove in this section the following result:

\begin{theorem}\label{thm:pontryagin}
	For every $n\geq 0$ there exists a surjectif homomorphism of abelian groups
	$$
	\mbf R\bds\Omega_{4n}^{SO}\to \R^{p(n)}
	$$
	where $p(n)$ is the number of partitions of $n$, i.e. the number of different ways to write $n$ as a sum of natural numbers.
\end{theorem}

The proof follows the scheme of the classical theorem of Thom-Pontryagin \cite{thom1954cobordism,milnorstasheff1974book} which states that a compact oriented smooth $4n$-manifold is the boundary of an compact oriented smooth $(4n+1)$-manifold if and only if its Pontryagin numbers vanish.

\subsection{Cohomology} Let $\mbf X$ be a random manifold with boundary. We define the \textit{cotangent space} of $\mbf X$ as the dual module 
$$
\Omega^{1}(\mbf X)=\op{Hom}_{\mc A(\mbf X)}(\tau(\mbf X),\mc A(\mbf X))
$$ 
and we define the exterior differential $d:\mc A(\mbf X)\to \Omega^1(\mbf X)$ by setting $da(\xi)=\xi(a)$ for every funtion $a\in \mc A(\mbf X)$ and tangent vector field $\xi\in \tau(\mbf X)$. The exterior derivative can be extended in the usual way to the whole exterior algebra 
$$
\Omega^{\bullet}(\mbf X)=\bigoplus_{i=0}^{\infty} \Omega^{i}(\mbf X)
$$ 
where $\Omega^i(\mbf X)=\bigwedge^i_{\mc A(\mbf X)}\Omega^1(\mbf X)$ is the $i$-th exterior power of $\Omega^1(\mbf X)$ over $\mc A(\mbf X)$, such that we have the usual identity $d^2=0$. The homology of the corresponding cochain complex is called the {\it de Rham cohomology} of $\mbf X$ and denoted by $H^{\bullet}(\mbf X)$. It has a natural real algebra structure inherited from the one in $\Omega^\bullet(\mbf X)$. In the other hand, we can define the $K$-theory of $\mbf X$, denoted by $K(\mbf X)$, as the algebraic $K$-theory of the algebra $\mc A(\mbf X)$. Recall that $K(\mbf X)$ is a ring structure induced by directed sums and tensor product of projectif modules. We will see that the $K$-theory and de Rham cohomology of bounded laminations can be seen as contravariant functors. For this, let $\mbf X$ and $\mbf Y$ be two random manifolds. A continuous mesurable map $f:\mbf X\to \mbf Y$ is said to be \textit{smooth} if for every $a\in\mc A(\mbf Y)$ the function $\mc A(f)(a)=a\circ f$ belongs to $\mc A(\mbf X)$. Indeed, Observe that the corresponding morphism of algebras $\mc A(f):\mc A(\mbf Y)\to \mc A(\mbf X)$, such that we can see $\mc A(\mbf X)$ as a $\mc A(\mbf Y)$-module. Observe that the map $E\mapsto E\otimes_{\mc A(\mbf Y)} \mc A(\mbf X)$ transforms finitely generated projectif $\mc A(\mbf Y)$-modules into finitely generated projectif $\mc A(\mbf X)$-modules and hence it induces a natural morphism of rings $K(f):K(\mbf Y)\to K(\mbf X)$. It remains the construction of a functorial homomorphism 
$$
\Omega^{\bullet}(f):\Omega^{\bullet}(\mbf Y)\to \Omega^{\bullet}(\mbf X)
$$
For this we need to prove the following result:
\begin{lemma}
	If $\mbf X$ is a bounded lamination, the $\mc A(\mbf X)$-module $\Omega^1(\mbf X)$ generated by the image of $d$. 
\end{lemma}
\begin{proof}
	Since $\Omega^1(\mbf X)=\sum_i a_i\Omega^1(U_i)$ for any trivialisation $\{U_i\}$ and any subordinated partition of unity $\{a_i\}$, one can assume that  $\mbf X$ is a trivial random manifold. In that case, it follows from the proof of lemma \ref{lem:trivial-free} that $\Omega^1(\mbf X)$ is generated (as $\mc A(\mbf X)$-module) by the differentials $dx_i$ where $x_i:\mbf X\to \R$ are the coordinate fonctions.  
\end{proof}
In particular the vector space $\Omega^n(\mbf X)$ is generated by the elements of the form $a_0\,da_1,\dots da_n$ and the exterior differential is defined recursively by the formula
$$
d(a_0\,da_1\cdots da_n)=da_0\,da_1\cdots da_n
$$
Hence, $\Omega^\bullet(f)$ can be defined as the unique morphism of differential algebras verifying $\Omega^0(f)=\mc A(f)$. Such a construction is clearly functorial.

\subsection{Chern-Weil theory}
We sketch here the Chern-Weil theory in the context of random manifolds. This construction was developed in the late 1940s by Shiing-Shen Chern and Andr\'e Weil in the wake of proofs of the generalized Gauss-Bonnet theorem \cite{chern44gaussbonnet} . This theory was an important step in the theory of characteristic classes and it allowed to clarify the relationship between cobordism, K-theory and usual cohomology. We shall follow the general construction given in the context of differential graded algebras over general rings (see for example Karoubi \cite{karoubi87cyclic} or Loday \cite{loday98bookcyclichomology}).

Let $(\mc A^\bullet,d)$ be any commutative graded differential real algebra and let $E$ be a finitely generated projectif $\mc A^0$-module. A connection in $E$ is a degree $1$ linear endomorphism $\nabla$ of $\mc A^\bullet\otimes E$ such that $\SC{\nabla,a}=da$ for every $x\in \mc A^\bullet$, where $\SC{\cdot,\cdot}$ denotes the graded commutator. A straightforward computation shows that the degree $2$ operator $\nabla^2$ verifies $[\nabla^2,a]=0$ for every $a\in \mc A^2$, and it thus defines an element $R_\nabla\in \mc A^2\otimes\op{End}(E)$ called the curvature of $\nabla$. For every invariant polynomial $p:M_\infty(\R)\to \R$, i.e. any polynomial function verifying $p(AB)=p(BA)$ for every pair of real matrices $A$ and $B$, and for every finitely generated module $E$ over $\mc A^0$, we have a natural map $p:\mc A^\bullet\otimes \op{End}(E)\to \mc A^\bullet$. We can define
$$
p(E,\nabla)=p(R_\nabla)\in \mc A^{2\bullet}
$$
A standard homotopy argument shows that $p(E,\nabla)$ is closed and its cohomology class does not depend on $\nabla$. Hence it defines an element $p(E)\in H^{2\bullet}(A^{\bullet},d)$.

\begin{example}[The Chern character]
	Let $n\in \N$. The \textit{Chern $n$-character} of $E$ with respect to $(\mc A^\bullet,d)$ is defined as $ch_n(E)$ where
	$$
	ch_n(A)=\frac 1{n!}\op{tr}(A^n)
	$$
	It verifies the following properties:
\begin{enumerate}
  \item $ch_n(E\oplus E')=ch_n(E)+ch_n(E')$.
  \item $ch_n(E\otimes E')=ch_n(E)ch_n(E')$.
\end{enumerate}
and hence induces a homomorphism of rings
$$
ch_n:K(\mc A^0)\to H^{2\bullet}(\mc A^\bullet,d)
$$
\end{example}

\begin{example}[The total Pontryagin class]
		The total Pontryagin class of $E$ with respect to $(\mc A^\bullet,d)$ is defined as $\bds p(E)$ where
	$$
	\bds p(A)=\det(1+A)
	$$
	It verifies the following properties
\begin{enumerate}
  \item $\bds p(\mc A^0)=1$.
  \item $\bds p(E\oplus E')=\mbf p(E)\mbf p(E')$.
\end{enumerate}

\end{example}

\subsection{The fundamental class of a random manifold} Let $V\subset \R^n$ be an open set and $\mbf K$ a standard probability space. Since the tangent space $\tau(V\otimes \mbf K)$ is a rank $n$ free $\mc A(V\otimes \mbf K)$-module, so is $\Omega^1(V\otimes \mbf K)$, and we thus have a natural identification of graded algebras
$$
\Omega^\bullet(V\otimes\mbf K)\simeq C^{\infty}(V,L^{\infty}(\mbf K))\otimes \wedge^\bullet \R^n
$$
By the dominated convergence theorem, integration through $\mbf K$ gives rise to a linear map $\int_{\mbf K}:C^{\infty}(V,L^{\infty}(\mbf K))\to C^{\infty}(V)$ which extends to a linear morphism of graded differential algebras $\int_{\mbf K}:\Omega^\bullet(V\otimes \mbf K)\to \Omega^\bullet(V)$ 
where $\Omega^\bullet(V)$ is the algebra of usual differential forms on $V$. Let $\mbf X$ be an oriented random $n$-manifold together with an oriented compact atlas $\{\varphi_i:U_i\to V_i\times \mbf K_i\}$ and a subordinated partition of unity $a_i\in \mc A(U_i)$. Fot every $\omega\in \Omega^{n}(\mbf X)$ and every $a_i$ we can see $a_i\omega$ as an element of $\Omega^{n}(V_i\otimes \mbf K_i)$ and $\int_{\mbf K_i}a_i\omega$ as a smooth $n$-form with compact support on $V_i$. We shall define
$$
\int_{\mbf X} \omega =\sum_i\int_{V_i}\int_{\mbf K_i} a_i\omega 
$$
A straightforward computation shows that this definition does not depend on the partition of unity nor the atlas. For the sake of simplicity, for every $\omega\in \Omega^{\bullet}(\mbf X)$ we shall set $\int_{\mbf X}\omega=\int_{\mbf X}\omega_n$, where $\omega_n$ denotes the component of $\omega$ in $\Omega^{n}(\mbf X)$. This defines a linear map
$$
\int_{\mbf X}:\Omega^{\bullet}(\mbf X)\to \R
$$
called the \textit{fundamental cycle} of $\mbf X$. The following result is a direct consequence of the definition of the above integral and the well known Stokes theorem on $\R^n$:

\begin{proposition}\label{prop:stokes}
	Let $\mbf X$ an oriented random $n$-manifold with boundary and let $\partial\mbf X$ its boundary. Then for every $\omega\in \Omega^{\bullet}(\mbf X)$ we have
	$$
	\int_{\mbf X}d\omega=\int_{\partial\mbf X}\Omega^{\bullet}(i)(\omega)
	$$
	where $i:\partial\mbf X\to\mbf X$ is the inclusion map.
In particular, since $\partial\mbf X$ has no boundary, the fundamental cycle of $\partial\mbf X$ induces a linear map $\int_{\partial\mbf X}:H^{\bullet}(\partial\mbf X)\to \R$ called the \emph{fundamental class} of $\partial\mbf X$. For every $\alpha\in H^\bullet(\mbf X)$ one has
	$$
	\int_{\partial\mbf X}H^\bullet(i)(\alpha)=0
	$$
\end{proposition}

\subsection{Pontryagin numbers}
The aim of this paragraph is to give a proof of theorem \ref{thm:pontryagin}. As in the classical case of compact manifolds, the proof relies in the cobordism invariance of characteristic numbers, which are obtained as a pairing between the fundamental class of the lamination and some suitable cohomology classes. Since the tangent space of a random manifold $\mbf X$ is a projectif $\mc A(\mbf X)$-module, we can define the \textit{total Pontryagin class} of $\mbf X$, noted $\bds p(\mbf X)$, as the total Pontryagin class of $\tau(\mbf X)$ with respect to the differential algebra $(\Omega^{\bullet}(\mbf X),d)$. The following lemma shows the naturality of the total Pontryagin class with respect to bordisms:

\begin{lemma}\label{lem:pontryagin-natural}
	If $\mbf X$ is a random $n$-manifold with boundary and $i:\partial\mbf X\to\mbf X$ is the inclusion of its boundary, then
	$$
	\bds p(\partial\mbf X) = H^\bullet(i)(\bds p(\mbf X))
	$$
\end{lemma}
\begin{proof}
	Since the total Pontryagin class is additive with respect to direct sums and $\bds p(\mc A(\mbf X))=1$, it suffices to prove that there exists an isomorphism of $\mc A(\mbf X)$-modules
	$$
	\tau(\mbf X)\otimes \mc A(\partial\mbf X)\simeq \tau(\partial\mbf X)\oplus \mc A(\partial\mbf X)
	$$
Let $g$ be a Riemannian metric on $\mbf X$ and let $\{\varphi:U_i\to V_i\otimes \mbf K_i\}_{i\in I}$ an oriented atlas together with a subordinated partition of unity $\alpha_i$. For every $i\in I$ we shall note $g_i$ the induced riemannian metric on $U_i$. 	By the lemma \ref{lem:trivial-free} we have identifications $\tau(U_i)=\mc A(U_i)^n$ and $\tau(\partial U_i)=\mc A(\partial U_i)^{n-1}$, which give rise to natural isomorphisms of $\mc A(\mbf X)$-modules
	$$
	\tau(U_i)\otimes \mc A(\partial U_i)\simeq \tau(\partial U_i)\oplus \mc A(\partial U_i)	
	$$
	The induced inclusions $\tau(\partial u_i)\subset \tau(U_i)\otimes \mc A(\partial U_i)$ add up into an inclusion
	$$
	\tau(\partial \mbf X)\subset \tau(\mbf X)\otimes \mc A(\partial \mbf X)
	$$
Let $\hat g_i$ be the inner product on $\tau(U_i)\otimes \mc A(\partial U_i)$ given by 
	$$
	\hat g_i(\xi\otimes a\otimes \nu\otimes b)=g_i(\xi a\otimes \nu b).
	$$
 and let $\tau(\partial U_i)^\perp$ be the orthogonal complement of $\tau(\partial U_i)$ with respect to $\hat g_i$, which is obviously isomorphic to $\mc A(\partial U_i)$. Let $\xi_i$ be a $\hat g$-normal generator of $\tau(\partial U_i)^\perp$ as $\mc A(\partial\mbf X)$-module. Since the atlas is oriented, one can construct mesurable maps $\epsilon:\mbf K_i\to \{\pm 1\}$ such that the vector $\xi=\sum_i\alpha_i \epsilon_i\xi_i\in \tau(\mbf X)\otimes \mc A(\partial\mbf X)$ does not vanish. A direct computation shows that  
	$$
	\mc A(\partial \mbf X)\xi = \tau(\partial \mbf X)^{\perp}
	$$
	which finishes the proof. 
\end{proof}

Let $\mbf Y$ be an oriented random $4n$-manifold without boundary. Let us note $\bds p_{k}(\mbf Y)\in H^{4k}(\mbf Y)$ the component of degree $4k$ of the total Pontryagin class of $\mbf Y$. It follows from the lemma above that $\bds p_{k}(\mbf Y)=H^\bullet(i)(\bds p_{k}(\mbf Y))$ for every $k$. A partition of $n$ is a set $\alpha\subset \N^{*}$ of strictly positive integers such that $\sum_{k\in \alpha} k=n$. We shall note $\mc P(n)$ the set of partitions of $n$. For every $\alpha\in \mc P(n)$ we shall note  
$$
\bds p_\alpha(\mbf Y)=\prod_{k\in\alpha}\bds p_{k}(\mbf Y)\in H^{2n}(\mbf Y)
$$
The real numbers $p_\alpha(Y):=\int_{\mbf Y}\bds p_\alpha(\mbf Y)$ are called the {\em Pontryagin numbers} of $\mbf Y$. Let us assume that $\mbf Y$ is the boundary of an oriented random manifold $\mbf X$ and let $i:\mbf Y\to \mbf X$ be the inclusion. Since $H^\bullet(i)$ is a morphism of graded rings, it follows from lemma \ref{lem:pontryagin-natural} that $\bds p_{\alpha}(\mbf Y)=H^{\bullet}(i)(\bds p_\alpha(\mbf X))$,   The proposition \ref{prop:stokes} implies that all the Pontryagin numbers of $\mbf Y$ vanish. Moreover, one easily sees that $p_{\alpha}(-\mbf Y)=-p_\alpha(\mbf Y)$ and $p_{\alpha}(\mbf Y+\mbf Y')=p_\alpha(\mbf Y)+p_\alpha(\mbf Y')$. Hence, it follows that the map 
$\mbf Y\mapsto p_{\bullet}(\mbf Y)$ induces a morphism of abelian groups 
$$
\Phi:\mbf R\bds \Omega^{SO}_{4n}\to \R^{\mc P(n)}
$$
To complete the proof of the theorem \ref{thm:pontryagin}, it remains to show that $\Phi$ is surjective. Let us note $M_\alpha=\prod_{k\in \alpha}\mbf P^{2k}(\C)$ for $\alpha\in \mc P(n)$. It is a well known fact (see for instance \cite{milnorstasheff1974book}) that for every $\alpha$ the matrix $A=(p_{\beta}(M_\alpha))_{\alpha,\beta}$ is non singular. In particular, for every $v\in \R^{\mc P(n)}$ there exists $\lambda\in \R^{\mc P(n)}$ such that
$$
v=\sum_\alpha p_{\bullet}(M_\alpha\otimes \mbf K_{\lambda_{\alpha}})=\Phi\left[\sum_\alpha M_\alpha\otimes \mbf K_{\lambda_{\alpha}}\right]
$$
where $\mbf K_\lambda$ is the standard diffuse measure space of mass $\lambda$. This completes the proof. 

\bibliography{../mybib}
\bibliographystyle{alphaurl}

\end{document}